\documentclass[11pt]{article}

\usepackage{fullpage}

\usepackage[leqno]{amsmath}
\usepackage{amssymb}
\usepackage{amsthm}
\usepackage{enumerate}
\usepackage{verbatim}
\usepackage{mathrsfs}
\usepackage{dsfont}
\usepackage{graphicx}

\newcommand{\ds}{\displaystyle}
\newcommand{\ts}{\textstyle}

\newcommand{\mc}{\mathcal}

\newcommand{\qbinom}[2]{\genfrac{[}{]}{0pt}{}{#1}{#2}}

\newcommand{\eps}{\varepsilon}

\newcommand{\wh}{\widehat}
\newcommand{\BAR}[1]{\overline{#1}}

\newcommand{\probabilityfont}[1]{\mathsf{#1}}

\let\Pr\undefined

\DeclareMathOperator*{\Pr}{\probabilityfont{Pr}}

\newcommand{\defeq}{\stackrel{\text{\tiny{def}}}{=}}

\newcommand{\fieldfont}[1]{\mathbb{#1}}

\newcommand{\R}{\fieldfont{R}}

\newcommand{\Q}{\fieldfont{Q}}
\newcommand{\F}{\fieldfont{F}}

\renewcommand{\ge}{\geqslant}
\renewcommand{\le}{\leqslant}

\newtheoremstyle{theorem-style}
  {}
  {}
  {\slshape}
  {}
  {\scshape}
  {.}
  {.5em}
  {}

\newtheorem{thm}{Theorem}
\newtheorem*{thm*}{Theorem}

\newtheorem{la}[thm]{Lemma}
\newtheorem{cor}[thm]{Corollary}

\theoremstyle{definition}
\newtheorem{df}[thm]{Definition}

\usepackage{mathtools}
\renewcommand{\defeq}{\vcentcolon=}

\def\[#1\]{\begin{align*}#1\end{align*}}

\renewcommand{\qbinom}[2]{{{\genfrac{[}{]}{0pt}{}{#1}{\smash{#2}}}_q}}

\renewcommand{\L}[1]{\mc L_q(#1)}
\renewcommand{\P}[1]{\mc P(#1)}

\renewcommand{\Q}{Q}
\renewcommand{\partial}{\triangle}

\theoremstyle{theorem}
\newtheorem*{mainthm*}{Theorem \ref{thm:density} (restated)}{}

\begin{document}

\title{Thresholds in the Lattice of Subspaces of $(\F_q)^n$}
\author{Benjamin Rossman\\ University of Toronto}
\date{\today}
\maketitle{}

\begin{abstract}
Let $\Q$ be an ideal (downward-closed set) in the lattice of linear subspaces of $(\F_q)^n$, ordered by inclusion. For $0 \le k \le n$, let $\mu_k(\Q)$ denote the fraction of $k$-dimensional subspaces that belong to $\Q$. We show that these densities satisfy
\[
  \mu_k(\Q) = \frac{1}{1+z}
  \quad \Longrightarrow\quad
  \mu_{k+1}(\Q) \le \frac{1}{1+qz}.
\]
This implies a sharp threshold theorem: if $\mu_k(\Q) \le 1-\eps$, then $\mu_\ell(\Q) \le \eps$ for $\ell = k + O(\log_q(1/\eps))$.
\end{abstract}

\section{Introduction}

Let $\L{n}$ be the lattice of linear subspaces of $(\F_q)^n$, ordered by inclusion.
Let $\Q$ be a nontrivial ideal in $\L{n}$ (that is, a nonempty proper subset of $\L{n}$ such that $A \in \Q$ implies $B \in \Q$ for all $B \subset A$). For $0 \le k \le n$, let $\mu_k(\Q)$ denote the fraction of $k$-dimensional subspaces that belong to $\Q$. Densities $\mu_k(\Q)$ are known to be non-increasing: thus,
\[
  1 = \mu_0(\Q) \ge
  \cdots \ge \mu_{t-1}(\Q) \ge 1/2 > \mu_t(\Q) \ge \cdots 
  \ge \mu_n(\Q) = 0
\]
for a unique $t$.
This paper addresses the question: How quickly must $\mu_k(\Q)$ transition from $1-o(1)$ to $o(1)$?

It follows from known results (described in \S\ref{sec:KK-BT}) that
\[
  \mu_{\lfloor (t-1)/c \rfloor}(\Q) \ge 2^{-1/c}
  \quad\text{ and }\quad
  \mu_{\lceil ct \rceil}(\Q) \le 2^{-c}
\]
for all $c \ge 1$. This is the $q$-analog of the Bollob\'as-Thomason Theorem \cite{bollobas1987threshold}, which speaks of ideals in the boolean lattice $\P{n}$ of subsets of $\{1,\dots,n\}$. 

On the one hand, $\L{n}$ is the $q$-analog of $\P{n}$; 
on the other hand, it is a sub-lattice of $\P{q^n}$. 
This raises the question: Do $k$-subspace densities of ideals in $\L{n}$ scale like $k$-subset densities in $\P{n}$ or like $q^k$-subset densities in $\P{q^n}$?
Quantitatively, the latter suggests we should expect that
\[
  \mu_{t-1-c}(\Q) \ge 1 - q^{-c}
  \quad\text{ and }\quad
  \mu_{t+c}(\Q) \le q^{-c}.
\]
for all integers $c \ge 1$. 
This is precisely what we show.

Our main result actually concerns shadows in the subspace lattice. Let $\L{n,k}$ denote the set of $k$-dimensional subspaces of $(\F_q)^n$. For $1 \le k \le n$ and $S \subseteq \L{n,k}$, the {\em shadow} of $S$ is the set $\partial S \subseteq \L{n,k-1}$ defined by
$
  \partial S \defeq \{B \in L_{n,k-1} : \exists A \in S,\ A \subset B\}$.
We show:
\begin{thm}\label{thm:density}
For all $1 \le k \le n$ and $S \subseteq \L{n,k}$, if $\mu_k(S) = (1+z)^{-1}$ where $z \in \R_{\ge 0}$, then
\[
  \mu_{k-1}(\partial S) \ge \left(1+
  \frac{q(q^{k-1}-1)(q^{n-k}-1)}
          {(q^k-1)(q^{n-k+1}-1)}
  \cdot z \right)^{-1} 
  \ge \left(1+\frac{z}{q}\right)^{-1}.
\]
\end{thm}

The first inequality in Theorem \ref{thm:density} is tight in two cases:
\begin{itemize}
\item
when $S$ is the set of $k$-dimensional subspaces of a fixed $n{-}1$-dimensional space ($z = \frac{q^n-q^{n-k}}{q^{n-k}-1}$), as well as
\item
when $S$ is the set of $k$-dimensional subspaces not containing a fixed $1$-dimensional space ($z = \smash{\frac{q^k-1}{q^n-q^k}}$).
\end{itemize}
For values of $z$ between $\frac{q^k-1}{q^n-q^k}$ and $\frac{q^n-q^{n-k}}{q^{n-k}-1}$, Theorem \ref{thm:main} improves the lower bound on $\mu_{k-1}(\partial S)$ given by a $q$-analog of the Kruskal-Katona Theorem due to Chowdhury and Patk\'os \cite{chowdhury2010shadows}. 

A sharp threshold theorem for $\L{n}$ follows immediately from Theorem \ref{thm:density} and the observation that $\partial(Q \cap \L{n,k}) \subseteq Q \cap \L{n,k-1}$ for ideals $Q$.

\begin{thm}\label{thm:main}
For every ideal $\Q$ in $\L{n}$ and $1 \le k \le n-1$, if $\mu_k(\Q) = (1+z)^{-1}$, then $\mu_{k-1}(\Q) \ge (1+(z/q))^{-1}$ and $\mu_{k+1}(\Q) \le (1+qz)^{-1}$.
As a consequence, if $\mu_k(\Q) \le 1-\eps$, then $\mu_\ell(\Q) \le \eps$ for $\ell = k + O(\log_q(1/\eps))$.
\end{thm}

The rest of the paper is organized as follows. In \S\ref{sec:KK-BT} we describe the previous $q$-analogs of the Kruskal-Katona and Bollob\'as-Thomason Theorems and their dual versions. In \S\ref{sec:main} we prove Theorem \ref{thm:density} using well-known tools (the Expander Mixing Lemma and bounds on the eigenvalues of Grassmann graphs). In \S\ref{sec:tightness} we discuss the tightness of the results. Finally, in \S\ref{sec:query} we give an application of Theorem \ref{thm:main} to a problem in query complexity.

\section{$q$-Analogs of Kruskal-Katona and Bollob\'as-Thomason}\label{sec:KK-BT}

For $x \in \R_{\ge 0}$, let $[x]_q \defeq \frac{q^x-1}{q-1}$. The (Gaussian) $q$-binomial coefficient $\qbinom{x}{k}$ is defined by
$$
  \qbinom{x}{k}
  \defeq 
  \prod_{i = 0}^{k-1} \frac{[x-i]_q}{[k-i]_q}.
$$
Note that $[0]_q=0$ and $[1]_q=1$ and $|\L{n,k}| = \qbinom{n}{k} = \qbinom{n}{n-k}$ for integers $n \ge k$.

Chowdhury and Patk\'os \cite{chowdhury2010shadows} proved a $q$-analog the Kruskal-Katona Theorem \cite{katona2009theorem,kruskal1963number}, specifically a version due Keevash \cite{keevash2008shadows}. (See \cite{wang2011intersecting} for an alternative proof.)

\begin{thm}[$q$-Kruskal-Katona]\label{thm:qKK}
For all $1 \le k \le n$ and $S \subseteq \L{n,k}$, if $|S| = \qbinom{x}{k}$, then $|\partial S| \ge \qbinom{x}{k-1}$.
Moreover, this bound is tight when $S$ is the set of $k$-dimensional subspaces of a fixed $\ell$-dimensional space where $k \le \ell \le n$.
\end{thm}

Note that the parameter $n$ (the dimension of the ambient vector space) plays no role in this bound, in contrast to Theorem \ref{thm:density}. It turns out Theorem \ref{thm:qKK} is slack when $n-1 < x < n$; this is precisely where Theorem \ref{thm:density} gives an improvement (as we discuss in \S\ref{sec:tightness}).

Combining Theorem \ref{thm:qKK} with 
the inequality 
$({\qbinom{x}{k-1}}/{\qbinom{n}{k-1}})^k \ge ({\qbinom{x}{k}}/{\qbinom{n}{k}})^{k-1}$
for all $k \le x \le n$, we have the following $q$-analog of the Bollob\'as-Thomason Theorem \cite{bollobas1987threshold} for the boolean lattice $\P{n}$.

\begin{thm}[$q$-Bollob\'as-Thomason]\label{thm:qBT}
For every ideal $\Q$ in $\L{n}$, 
\[\mu_1(\Q) \ge \mu_2(\Q)^{1/2} \ge \mu_3(\Q)^{1/3} \ge \cdots \ge \mu_n(\Q)^{1/n}.\]
In particular, if $\mu_{t-1}(\Q) \ge 1/2 > \mu_t(\Q)$, then
$\mu_{\lfloor (t-1)/c \rfloor}(\Q) \ge 2^{-1/c}$ and
$\mu_{\lceil ct \rceil}(\Q) \le 2^{-c}$ for all $c \ge 1$.
\end{thm}

If we regard $\Q$ as a sequence of ideals in $\L{n}$, one for each $n$, then Theorem \ref{thm:qBT} implies that every nontrivial $\Q$ has a {\em threshold function} $t(n)$, meaning that $\mu_{k(n)}(\Q) = 1-o(1)$ for all $k(n) = o(t(n))$ and $\mu_{\ell(n)}(\Q) = o(1)$ for all $\ell(n) = \omega(t(n))$. In the boolean lattice $\P{n}$, nothing more can be said in general, although certain classes of ideals in $\P{n}$, such as monotone graph properties when $n = \binom{m}{2}$, are known to have {\em sharp thresholds} such that $\mu_{k(n)}(\Q) = 1 - o(1)$ and $\mu_{\ell(n)}(\Q) = o(1)$ for some $k(n) = t(n) - o(t(n))$ and $\ell(n) = t(n) + o(t(n))$ (see \cite{friedgut1996every}). 
In the same sense, Theorem \ref{thm:main} shows that every sequence of nontrivial ideals in $\L{n}$ has a sharp threshold.

\subsection{Dual versions of Theorems \ref{thm:qKK} and \ref{thm:qBT}}

For a subspace $A$ of $(\F_q)^n$, the orthogonal complement is defined by $$
  A^\perp \defeq \{b \in (\F_q)^n : \ts\sum_{i=1}^n a_ib_i = 0 \text{ for all } a \in A\}.$$
Note that $\dim(A^\perp) = n-\dim(A)$ and $(A^\perp)^\perp = A$ and $B \subseteq A \Longrightarrow A^\perp \subseteq B^\perp$.

For every ideal $Q$ in $\L{n}$, there is a dual ideal $Q^\ast \defeq \{A \in \L{n} : A^\perp \notin \Q\}$ satisfying $\mu_k(Q^\ast) = 1 - \mu_{n-k}(Q)$. 
Applying Theorem \ref{thm:qBT} to $Q^\ast$ yields:

\begin{thm}[Dual $q$-Bollob\'as-Thomason]\label{thm:dual-qBT}
For every ideal $Q$ in $\L{n}$,
\[
1-\mu_{n-1}(\Q) \ge (1-\mu_{n-2}(\Q))^{1/2} \ge (1-\mu_{n-3}(\Q))^{1/3} \ge \cdots \ge (1-\mu_0(\Q))^{1/n}.
\]
In particular,
$\mu_{\lfloor c(t-1) + (1-c)n \rfloor}(\Q) \ge 1 - 2^{-c}$ and
$\mu_{\lceil t/c + (1-1/c)n \rceil}(\Q) \le 1 - 2^{-1/c}$ 
for all $c \ge 1$. (This improves Theorem \ref{thm:qBT} when $t \ge n/2$.)
\end{thm}
 
Similarly, there is a dual version of Theorem \ref{thm:qKK}. It may be helpful to include the proof, since we will use a similar argument in \S\ref{sec:main}.

\begin{thm}[Dual $q$-Kruskal-Katona]\label{thm:dual-qKK}
For all $1 \le k \le n$ and $n-k+1 \le y \le n$ and $S \subseteq \L{n,k}$, if $|S| = \qbinom{n}{k}-\qbinom{y}{n-k}$, then $|\partial S| \ge \qbinom{n}{k-1}-\qbinom{y}{n-k+1}$. 
\end{thm}

\begin{proof}
We will assume $|\partial S| < \qbinom{n}{k-1}-\qbinom{y}{n-k+1}$ and prove that $|S| < \qbinom{n}{k}-\qbinom{y}{n-k}$.
Define $T \subseteq \L{n,n-k+1}$ by
\[
  T \defeq \{B^\perp : B \in \L{n,k-1} \setminus \partial S\}.
\]
Note that $|T| = \qbinom{n}{k-1} - |\partial S| = \qbinom{y}{n-k+1}$. Therefore, Theorem \ref{thm:qKK} implies $|\partial T| > \qbinom{y}{n-k}$.

For all $A \in \L{n,k}$, observe that
\[
  A^\perp \in \partial T 
  \Longleftrightarrow 
  {}\mbox{}&\exists B \in \L{k-1} \setminus \partial S,\ A^\perp \subset B^\perp\\
  \Longleftrightarrow 
  {}\mbox{}&\exists B \in \L{k-1} \setminus \partial S,\ B \subset A\\
  \Longrightarrow 
  {}\mbox{}&A \notin S.
\]
Therefore,
$
  S \subseteq \{A \in \L{n,k} : A^\perp \notin \partial T\}.
$
We conclude that $|S| = \qbinom{n}{k} - |\partial T| < \qbinom{n}{k} - \qbinom{y}{n-k}$, as required.
\end{proof}

\section{Proof of Theorem \ref{thm:density}}\label{sec:main}

The proof of Theorem \ref{thm:density} involves bounding the edge-expansion of sets in the Grassmann graph $J_q(n,k)$. We state the required definitions and lemmas below. (See \cite{subhash2018pseudorandom} for a much deeper study of expansion of Grassman graphs.)

\begin{df}
For a $d$-regular graph $G = (V,E)$ and $S \subseteq V$, the {\em edge-expansion} of $S$ is defined by
\[
  \Phi_G(S) \defeq \frac{|E(S,\BAR S)|}{d|S|}
\]
where $E(S,\BAR S)$ is the set of edges between $S$ and $\BAR S = V \setminus S$.
\end{df}

\begin{la}[Expander Mixing Lemma \cite{alon1988explicit}]\label{la:EML}
Let $G = (V,E)$ be a $d$-regular graph and suppose the second largest eigenvalue (in absolute value) of the adjacency matrix of $G$ is at most $\lambda$. Then for all $S \subseteq V$,
\[
  \left(1-\frac{\lambda}{d}\right)\left(1-\frac{|S|}{|V|}\right)
  \le
  \Phi_G(S) 
  \le 
  \left(1+\frac{\lambda}{d}\right)\left(1-\frac{|S|}{|V|}\right).
\]
\end{la}

\begin{df}
For $1 \le k \le n$, the {\em Grassmann graph} $J_q(n,k)$ is the $q[k]_q[n-k]_q$-regular graph with vertex set $\L{n,k}$ and edge set
\[
  E_{J_q(n,k)} \defeq 
  \{(A_1,A_2) \in \L{n,k} \times \L{n,k} : \dim(A_1 \cap A_2) = k-1\}.
\]
\end{df}

\begin{la}[Spectrum of $J_q(n,k)$ \cite{brouwer2012distance}]\label{la:spectrum}
The adjacency matrix of $J_q(n,k)$ has eigenvalue
$
  q^{i+1}[k-i]_q[n-k-i]_q-[i]_q
$
with multiplicity $\qbinom{n}{i} - \qbinom{n}{i-1}$ for each $0 \le i \le \min(k,n-k)$. 
In particular, the second largest eigenvalue (in absolute value) equals $1$ if $k \in \{1,n-1\}$ and equals $q^2[k-1]_q[n-k-1]_q-1$ if $2 \le k \le n-2$.
\end{la}

Lemmas \ref{la:EML} and \ref{la:spectrum} give the following lower bound on $\Phi_{J_q(n,k)}(S)$.

\begin{la}\label{la:lower}
For all $2 \le k \le n-2$ and 
$S \subseteq \L{n,k}$,
\[
  \Phi_{J_q(n,k)}(S) \ge 
  \frac{[n]_q}{q[k]_q[n-k]_q}
  (1 - \mu_k(S)).
\]
\end{la}

\begin{proof}
Lemma \ref{la:EML} implies the lower bound
\[
  \Phi_{J_q(n,k)}(S) \ge
  \left(1-\frac{q^2[k-1]_q[n-k-1]_q-1}{[k]_q[n-k]_q}\right)\left(1-\mu_k(S)\right).
\]
By a straightforward calculation,
\[
  1 - \frac{q^2[k-1]_q[n-k-1]_q-1}{q[k]_q[n-k]_q}
  =
  \frac{q^{n+1}-q^n-q+1}{q^{n+1}-q^{k+1}-q^{n-k+1}+q}
  &=
  \frac{[n]_q}{q[k]_q[n-k]_q}.\qedhere
\]
\end{proof}

We next show an upper bound on $\Phi_{J_q(n,k)}(S)$ in terms of the ratio $\mu_k(S)/\mu_{k-1}(\partial S)$.

\begin{la}\label{la:upper}
For all $1 \le k \le n$ and 
$\emptyset \subset S \subseteq \L{n,k}$,
\[
  \Phi_{J_q(n,k)}(S) 
  \le 
  \frac{[n-k+1]_q}{q[n-k]_q}\left(1 - \frac{\mu_k(S)}{\mu_{k-1}(\partial S)}\right).
\]
\end{la}

\begin{proof}
For $B \in \partial S$, let 
$S_B \defeq \{A \in S : B \subset A\}.$
We have $\sum_{B \in \partial S}|S_B| = [k]_q|S|$ and, by the Cauchy-Schwarz inequality,
\[
  \sum_{B \in \partial S} |S_B|^2
  \ge
  \frac{(\sum_{B \in \partial S}|S_B|)^2}{|\partial S|}
  &=
  \frac{([k]_q|S|)^2}{|\partial S|}.
\]
Therefore,
\[
  |E_{J_q(n,k)}(S,\BAR S)| 
  =
  \sum_{B \in \partial S} |S_B \times \BAR{S}_B|
  &=
  \sum_{B \in \partial S} |S_B|\left([n-k+1]_q - |S_B|\right)\\
  &\le
  [k]_q|S|\left([n-k+1]_q - \frac{[k]_q|S|}{|\partial S|}\right).
\]
We now have
\[
  \Phi_{J_q(n,k)}(S) 
  =
  \frac{|E_{J_q(n,k)}(S,\BAR S)|}{q[k]_q[n-k]_q|S|}
  &\le
  \frac{[n-k+1]_q}{q[n-k]_q} - 
  \frac{[k]_q}{q[n-k]_q}\cdot
  \frac{|S|}{|\partial S|}.
\]
The lemma now follows from the equality
\[
  \frac{[k]_q}{q[n-k]_q}\cdot
  \frac{|S|}{|\partial S|}
  =
  \frac{[k]_q\qbinom{n}{k}}{q[n-k]_q\qbinom{n}{k-1}}\cdot
  \frac{\mu_k(S)}{\mu_{k-1}(\partial S)}
  &=
  \frac{[n-k+1]_q}{q[n-k]_q}\cdot
  \frac{\mu_k(S)}{\mu_{k-1}(\partial S)}.\qedhere
\]
\end{proof}

We are ready to prove:

\begin{mainthm*}
For all $1 \le k \le n$ and $S \subseteq \L{n,k}$, if $\mu_k(S) = (1+z)^{-1}$ where $z \in \R_{\ge 0}$, then
\[
  \mu_{k-1}(\partial S) \ge \left(1+
  \frac{q(q^{k-1}-1)(q^{n-k}-1)}
          {(q^k-1)(q^{n-k+1}-1)}
  \cdot z\right)^{-1} \ge \left(1+\frac{z}{q}\right)^{-1}.
\]
\end{mainthm*}

\begin{proof}
The second inequality is by a straightforward calculation:
\[
  \frac{q(q^{k-1}-1)(q^{n-k}-1)}
       {(q^k-1)(q^{n-k+1}-1)}
  &= \frac{1}{q}\left(1 - 
     \frac{(q-1)(q^{n-k+1}+q^k-q-1)}
          {(q^k-1)(q^{n-k+1}-1)}
          \right)
  \le \frac{1}{q}.
\]

For the first inequality, consider the case that $k \in \{1,n\}$. In both cases, we have $\mu_{k-1}(\partial S) = 1$ for every nonempty $S \subseteq \L{n,k}$. Therefore, the inequality holds (moreover, with equality since $[k-1]_q[n-k]_q = 0$).

Next, consider the case that $2 \le k \le n-2$. In this case, Lemmas \ref{la:lower} and \ref{la:upper} imply
\[
  \frac{[n]_q}{q[k]_q[n-k]_q}(1 - \mu_k(S))
  \le
  \Phi_{J_q(n,k)}(S) 
  \le
  \frac{[n-k+1]_q}{q[n-k]_q}\left(1 - \frac{\mu_k(S)}{\mu_{k-1}(\partial S)}\right).
\]
Therefore,
\[
  \frac{[n]_q}{[k]_q [n-k+1]_q}(1 - \mu_k(S))
  \le
  1 - \frac{\mu_k(S)}{\mu_{k-1}(\partial S)}.
\]
Substituting $(1+z)^{-1}$ for $\mu_k(S)$, this rearranges to
\[
  \mu_{k-1}(\partial S)
  \ge
  \left(
  1 + \left(1-\frac{[n]_q}{[k]_q[n-k+1]_q}\right)z
  \right)^{-1}
  =
  \left(
  1 + \frac{q(q^{k-1}-1)(q^{n-k}-1)}{(q^k-1)(q^{n-k+1}-1)}\cdot z
  \right)^{-1}.
\]

We derive the remaining case $k = n-1$ from the case $k=2$ via duality. Letting $S \subseteq \L{n,n-1}$, we will assume that
\[
  \mu_{n-2}(S) < \left(1+\frac{q(q^{n-2}-1)(q-1)}{(q^2-1)(q^{n-1}-1)}\cdot z\right)^{-1}
\]
and show that $\mu_{n-1}(S) < (1+z)^{-1}$.
Let $T \defeq \{B^\perp : B \in \L{n,n-2} \setminus \partial S\}$ and note that
\[
  \mu_2(T) = 1 - \mu_{n-2}(\partial S) 
  > 1 - \left(1+\frac{q(q^{n-2}-1)(q-1)}{(q^2-1)(q^{n-1}-1)} \cdot z\right)^{-1}
  =
  \left(1+\frac{(q^2-1)(q^{n-1}-1)}{z\cdot q(q^{n-2}-1)(q-1)}\right)^{-1}.
\]
From the case $k=2$, we have
\[
  \mu_1(\partial T) 
  \ge
  \left(1+\frac{q(q^{n-2}-1)(q-1)}{(q^2-1)(q^{n-1}-1)} 
  \cdot
  (\mu_2(T)^{-1}-1)
  \right)^{-1}
  > 
  \left(1+\frac{1}{z}\right)^{-1}.
\]
Since $S \subseteq \{A \in \L{n,n-1} : A^\perp \notin \partial T\}$ (as in the proof of Theorem \ref{thm:dual-qKK}), it follows that
\[
  \mu_{n-1}(S) \le 
  1 - \mu_1(\partial T) <
  1 - \left(1+\frac{1}{z}\right)^{-1} = (1+z)^{-1},
\]
as required.
\end{proof}

We remark that Theorem \ref{thm:density} is self-dual: for any $1 \le k \le n$ and $S \subseteq \L{n,k}$, we get the same inequality between $\mu_k(S)$ and $\mu_{k-1}(\partial S)$ as between $1 - \mu_{n-k}(\partial T)$ and $1 - \mu_{n-k+1}(T)$  where $T \defeq \{B^\perp : B \in \L{n,k-1} \setminus \partial S\}$.

\section{Tightness of the result}\label{sec:tightness} 

Fix a flag $V_0 \subset V_1 \subset \dots \subset V_n = (\F_q)^n$. (Without loss of generality, we may take $V_k = \{u \in (\F_q)^n : u_{k+1}=\dots=u_n=0\}$.)
For $1 \le j \le n$, let $Q_{\wh j}$ be the ideal
\[
  Q_{\wh j} \defeq 
  \{A \in \L{n} : A \cap 
  (V_j - V_{j-1})
  = \emptyset\}.
\]
In particular, $Q_{\wh 1}$ is the set of subspaces of $(\F_q)^n$ that do not contain $V_1$, while $Q_{\wh n}$ is the set of subspaces contained in $V_{n-1}$. It can be shown (by some tedious calculations) that
\begin{enumerate}[\quad(i)\ ]
\item
$\ds\mu_{n-j}(Q_{\wh j}) > 
1/2 > \mu_{n-j+1}(Q_{\wh j})$,
\item
if $2 \le k \le n-1$ and $\mu_k(Q_{\wh j}) = (1+z)^{-1}$, then 
$\ds\left(1+\frac{z}{q^2}\right)^{-1} \ge \mu_{k-1}(Q_{\wh j}) \ge \left(1+\frac{z}{q}\right)^{-1}$,
\item
$\ds\mu_k(Q_{\wh n}) 
  = 
  \frac{\qbinom{n-1}{k}}{\qbinom{n}{k}} 
  = 
  \frac{[n-k]_q}{[n]_q}
  = 
  \left(1 + 
  \frac{q^{n-k}(q^k-1)}{(q^{n-k}-1)}
  \right)^{-1}$,
\item
$\ds\mu_k(Q_{\wh 1}) 
  = 
  1 - \frac{\qbinom{n-1}{k-1}}{\qbinom{n}{k}} 
  = 
  1 - 
  \frac{[k]_q}{[n]_q}
  =
  \left(1 + 
  \frac{(q^k-1)}{q^k(q^{n-k}-1)}
  \right)^{-1}$.
\end{enumerate}

Inequalities (i) and (ii) show that Theorem \ref{thm:main} is essentially tight, no matter where in $\{1,\dots,n\}$ the threshold for $Q$ occurs. Equations (iii) and (iv) show that the first inequality of Theorem \ref{thm:density} is tight
\begin{itemize}
\item
when $S$ is the set of $k$-dimensional subspaces of a fixed $n{-}1$-dimensional space, as well as 
\item
when $S$ is the set of $k$-dimensional subspaces not containing a fixed $1$-dimensional space. 
\end{itemize}
The first example is also tight for $q$-Kruskal-Katona (Theorem \ref{thm:qKK}), while the second example is tight for the Dual $q$-Kruskal-Katona (Theorem \ref{thm:dual-qKK}).
Taking the maximum of the bounds given by Theorem \ref{thm:density}, \ref{thm:qKK} and \ref{thm:dual-qKK}, we get:

\begin{cor}\label{cor:combined}
For all $1 \le k \le n$ and $\emptyset \subset S \subset \L{n,k}$,
\[
  |\partial S| \ge
  \left\{
  \begin{array}{lllrl}
    \ds\qbinom{x}{k-1}
      &\ds\text{if } |S|=\qbinom{x}{k} 
      \text{ where } k \le x \le n-1,\\
   \vphantom{{}^{\Big|}{j}_{\Big|}}\ds\qbinom{n}{k-1}\left(1 + 
      \frac{z \cdot (q^{k-1}-1)}{q^{k-1}(q^{n-k+1}-1)}\right)^{-1}\!\!
      &\ds\text{if } |S|=\qbinom{n}{k}\left(1+\frac{z \cdot (q^k-1)}{q^k(q^{n-k}-1)}\right)^{-1}\!
      \text{ where } 
      1 \le z \le q^n,\\
    \ds\qbinom{n}{k-1} - \qbinom{y}{n-k+1} 
      &\ds\text{if } |S|=\qbinom{n}{k}-\qbinom{y}{n-k}
      \text{ where } 
      n-k+1 \le y \le n-1.
  \end{array}
  \right.
\]
\end{cor}

Corollary \ref{cor:combined} is known to be tight when $x$ or $y$ are integers (or $z \in \{1,q^n\}$, coinciding with cases $y=n-1$ and $x=n-1$).
In other cases, determining the optimal lower bound for $|\partial S|$ in terms of $|S|$ remains an open problem. 
In contrast, note that the original Kruskal-Katona Theorem \cite{katona2009theorem,kruskal1963number} completely solves the shadow minimization problem in the boolean lattice: if $S$ is a family of $k$-element sets and $|S| = \binom{n_k}{k} + \binom{n_{k-1}}{k-1} + \dots + \binom{n_j}{j}$ where $n_k > n_{k-1} > \dots > n_j = j \ge 1$, then $|\partial S| \ge \binom{n_k}{k-1}+\binom{n_{k-1}}{k-2}+\dots + \binom{n_j}{j-1}$ and this bound is tight. Moreover, a family of nested solutions is given by the subsets of $\{1,\dots,n\}$ in co-lexicographic order. The situation in $\L{n}$ appears more complicated, as nested solutions to the shadow minimization problem in $\L{n}$ are known not to exist \cite{bezrukov1999kruskal,harper1994isoperimetric,ure1996study}.

\section{Application to a query problem}\label{sec:query}

We conclude by giving an application of Theorem \ref{thm:main} to a problem in query complexity. In this problem, $A$ is a {\em hidden} nontrivial subspace of $(\F_2)^n$ and the goal is to learn a nonzero element of $A$ with probability $\ge 1/2$ by making $m$ simultaneous (non-adaptive) monotone queries. What is the minimum $m$ for which this is possible?  An upper bound of $O(n^2)$ is well-known (see \cite{kawachi2017query}). The following theorem gives a matching lower bound of $\Omega(n^2)$.

\begin{thm}\label{thm:query}
Let $(\Q_1,\dots,\Q_m)$ be a joint distribution over ideals in the subspace lattice of $(\F_2)^n$ and let $f$ be a function $\{0,1\}^m \to (\F_2)^n \setminus \{\vec 0\}$. Suppose that for every nontrivial subspace $A$ of $(\F_2)^n$, it holds that
\[
  \Pr_{\Q_1,\dots,\Q_m}[\ f(1_{\{A \in \Q_1\}},\dots,1_{\{A \in \Q_m\}}) \in A\ ] \ge 1/2
\]
where $1_{\{A \in \Q_i\}}$ is the indicator function for the event that $A \in \Q_i$. Then $m = \Omega(n^2)$.
\end{thm}

This lower bound follows from combining Theorem \ref{thm:main} with some lemmas from the paper \cite{kawachi2017query}, which proves the special case of Theorem \ref{thm:query} where ideals $Q_i$ are restricted to be of the form $\{A \in L_2(n) : A \cap U_i = \emptyset\}$ for subsets $U_i \subseteq (\F_2)^n$.

\bibliographystyle{plain}
\bibliography{thresholds.bib}

\end{document}